\documentclass[12pt,twoside,final]{article}
\usepackage{chngpage}
\usepackage[a4paper,left=22mm,right=22mm,top=30mm,bottom=30mm]{geometry}
\usepackage{imakeidx}
\usepackage{xr-hyper}
\usepackage{xr}
\usepackage[all]{xy}
\usepackage[center]{titlesec}
\titleformat*{\section}{\large\bfseries}
\usepackage{enumerate,cite}
\usepackage[latin1]{inputenc}
\usepackage{amssymb}
\usepackage{amsthm}
\usepackage{amsmath}
\usepackage{amsfonts}
\usepackage{mathrsfs}
\usepackage{hyperref}
\usepackage{color}
\usepackage{graphicx}
\usepackage{setspace}
\usepackage{bm}
\usepackage{enumitem}
\usepackage{booktabs}
\usepackage[capitalize]{cleveref}
\usepackage{tikz}
\usetikzlibrary{matrix,chains}
\usepackage[center]{titlesec}
\titleformat*{\section}{\normalsize\bfseries}
\usepackage{indentfirst}

\usepackage{amscd}

\newcommand{\IBr}{\operatorname{IBr}\nolimits}

\theoremstyle{remark}

\theoremstyle{definition}

\theoremstyle{plain}
\newtheorem{thm}{Theorem}[section]
\newtheorem{conj}[thm]{Conjecture}
\newtheorem{lem}[thm]{Lemma}

\newtheorem{cor}[thm]{Corollary}

\newtheorem{prop}[thm]{Proposition}

\newtheorem{conj*}{Conjecture}
\numberwithin{equation}{thm}

\usepackage{lastpage}
\usepackage{fancyhdr}

\begin{document}

\begin{center}{\Large\bf  On the Sum of Squares of Irreducible Brauer Character Degrees in Blocks
}

\bigskip{\large}

\bigskip{Hanxiao Li and Kun Zhang}

\bigskip{\scriptsize 1.School of Mathematics and Statistics, Central China Normal University, Wuhan, 430079, China
\\2.Faculty of Mathematics and Statistics, Hubei University, Wuhan, 430062, China}
\end{center}

{\noindent\small{\bf Abstract}
We study a weakened version of the Holm--Willems Local Conjecture.
The problem is reduced to quasi-simple groups under the assumption that the defect group is abelian.
Complete proofs are provided in the case \(p = 2\).
}

\medskip\noindent{\small{{{\bf Keywords} Brauer character; Block algebra; Abelian defect group

}
 }}

\section  {Introduction}

Let \(p\) be a prime number and \(k\) an algebraically closed field of characteristic \(p\).
Let \(G\) be a finite group and \(b\) a block of \(G\), that is, a central primitive idempotent of the group algebra \(kG\).
For the block \(b\), let \(\IBr(b)\) denote the set of irreducible Brauer characters of \(b\),
and let \(l(b) = |\IBr(b)|\).  In this paper, we mainly study the invariant
\[\tau(b) := \dim_k(kGb) / \sum_{\chi \in \IBr(b) }\chi (1)^2.\]

In \cite{HW}, Holm and Willems proposed an upper bound for \(\tau(b)\):

\begin{conj}\label{local-conj}
Let \(G\) be a finite group and \(b\) a block of \(G\) with defect group \(P\). Then
\(\tau(b)\leq l(b)|P|,\)
with equality holding if and only if \(l(b) = 1\).
\end{conj}

Regarding the bound of \(l(b)\), Malle and Robinson proposed a fundamental problem in \cite{MR}. Let \(s(b)\) denote the sectional \(p\)-rank of the defect group \(P\), which is defined as the largest rank of an elementary abelian section of \(P\). They stated the following conjecture.
\begin{conj}\label{MRconj}
    Let \(G\) be a finite group and \(b\) a block of \(G\), then \(l(b) \leq p^{s(b)}\).
\end{conj}

The aforementioned conjectures have been verified in certain cases; see references \cite{HW,Z,MR} for details.
Motivated by the two conjectures above, we consider the following natural question:
for a block \(b\) of \(G\) with non-trivial defect group \(P\), does the strict inequality
\begin{equation}\label{equation}
\tau(b) < p^{s(b)}|P|
\end{equation}
always hold?
Our main theorem reduces this problem for blocks with abelian defect groups to the case of quasi-simple groups.

\begin{thm}\label{Main}
If inequality (\ref{equation}) holds for all blocks of quasi-simple groups with non-trivial abelian defect groups, then it holds for all blocks with non-trivial abelian defect groups.
\end{thm}

As an application, we obtain the following corollary:

\begin{cor}\label{Main1}
Inequality (\ref{equation}) holds for all \(2\)-blocks with non-trivial abelian defect groups.
\end{cor}

\section{Proof of Theorem \ref{Main}}

In this section,
we will give a proof of Theorem \ref{Main}.
In general, we follow the notation and terminology of \cite{N1}.

Let \(G\) be a finite group and \(b\) a block of \(G\).
Let \(N\) be a normal subgroup of \(G\) and \(c\) a block of \(N\) covered by \(b\).
Denote by \(\widehat{G}\) be the stabilizer of \(c\) in \(G\).
Let \(\hat{b}\) be the unique block of \(\widehat{G}\) covering \(c\) such that \(\hat{b}^G = b\) (see \cite[Theorem 9.14]{N1}).

\begin{lem}\label{Fong}
Keep the notation above. Then $\tau(b) = \tau(\hat{b})$.
\end{lem}
\begin{proof}
By \cite[Theorem 9.14]{N1}, we obtain
\[
\operatorname{Irr}(b) = \left\{ \chi^G \mid \chi \in \operatorname{Irr}(\hat{b}) \right\} \quad \text{and} \quad \operatorname{IBr}(b) = \left\{ \varphi^G \mid \varphi \in \operatorname{IBr}(\hat{b}) \right\}.
\]
Let \((c_{\varphi \phi})_{\varphi,\phi \in \IBr(\hat{b})}\) and \( (c_{\varphi^G \phi^G})_{\varphi,\phi \in \IBr(\hat{b})} \) be the Cartan matrices of \(\hat{b}\) and \(b\), respectively. Then we have \(c_{\varphi \phi}=c_{\varphi^G \phi^G}\).
Recall that
\[
\dim_{k}(k\widehat{G}\hat{b})=\sum_{\varphi,\phi \in \IBr(\hat{b})} c_{\varphi \phi} \varphi(1) \phi(1),
\]
and a similar equality holds for $\dim_k(kGb)$. Consequently,
\[
\tau(b) = \frac{ \dim_{k}(kGb) }{ \sum_{\varphi \in \operatorname{IBr}(\hat{b})} \varphi^G(1)^2 } = \frac{ \dim_{k}(k\widehat{G}\hat{b})  }{ \sum_{\varphi \in \operatorname{IBr}(\hat{b})} \varphi(1)^2 } = \tau(\hat{b}).
\]
\end{proof}

\begin{lem}\label{Central}
Keep the notation above. Suppose that $N$ is a central $p$-subgroup of $G$. Let $P$ be a defect group of $b$, and let $\bar{b}$ be the block of $G/N$ dominated by $b$. Then
\[
\frac{\tau(b)}{p^{s(b)} |P|} \leq \frac{\tau(\bar{b})}{p^{s(\bar{b})} |P/N|}.
\]
\end{lem}
\begin{proof}
By \cite[Theorem 9.10]{N1}, we have $\IBr(\bar{b}) = \IBr(b)$, $P/N$ is a defect group of $\bar{b}$, and the Cartan matrices are related by $C_b = |N| \cdot C_{\bar{b}}$. Let $C_{\bar{b}} = (c_{\alpha\beta})_{\alpha, \beta \in \IBr(\bar{b})}$. Note that $s(b) \geq s(\bar{b})$, so $p^{s(b)} \geq p^{s(\bar{b})}$.
Therefore, we have
\[
\begin{aligned}
\frac{\tau(b)}{p^{s(b)} |P|}
&= \frac{|N|\sum_{\alpha,\beta\in \IBr(\bar{b})}c_{\alpha\beta}\alpha(1)\beta(1)}{p^{s(b)} |P|\sum_{\varphi \in \IBr(b) }\varphi (1)^2} \\
&= \frac{\tau(\bar{b})}{p^{s(b)} |P/N|} \\
&\leq \frac{\tau(\bar{b})}{p^{s(\bar{b})}|P/N|}.
\end{aligned}
\]
\end{proof}

\begin{lem}\label{IBr}
Keep the notation above. Suppose that \(G/N\) is a \(p'\)-group. Then, for any \(\varphi \in \IBr(N)\), we have
\[
\sum_{\chi \in \IBr(G|\varphi)} \chi(1)^2 = |G:G_{\varphi}| \cdot |G:N| \cdot \varphi(1)^2,
\]
where \(G_{\varphi}\) denotes the stabilizer of \(\varphi\) in \(G\).
\end{lem}

\begin{proof}
For any \(\chi \in \IBr(G|\varphi)\), let \(\ell_{\chi}\) be the multiplicity of \(\varphi\) in the restriction character \(\chi_N\) and \(f(\chi)\) the unique irreducible Brauer character of \(G_{\varphi}\) such that \(f(\chi)^G = \chi\).
By \cite[Theorem 8.9]{N1}, we have \(\ell_{\chi} = f(\chi)(1)/\varphi(1)\). Since \(|G:N|\) is coprime to \(p\), it follows from \cite[Corollary 8.7]{N1} that \(\varphi^G = \sum_{\chi \in \IBr(G|\varphi)} \ell_{\chi} \chi\).
Therefore, we have
\[
\begin{aligned}
\sum_{\chi \in \IBr(G|\varphi)} \chi(1)^2
&= |G:G_{\varphi}| \cdot \sum_{\chi \in \IBr(G|\varphi)} f(\chi)(1) \chi(1) \\
&= |G:G_{\varphi}| \cdot \varphi(1) \cdot \sum_{\chi \in \IBr(G|\varphi)} \ell_{\chi} \chi(1) \\
&= |G:G_{\varphi}| \cdot |G:N| \cdot \varphi(1)^2.
\end{aligned}
\]
\end{proof}

\begin{lem}\label{p-Reg}
Keep the notation above. If \(G/N\) is a \(p'\)-group, then \(\tau(b)=\tau(c)\).
\end{lem}

\begin{proof}
We proceed by induction on \(|G : N|\). By Lemma \ref{Fong}, for any \(N \leq H \trianglelefteq G\), we assume that the block of \(H\) covered by \(b\) is \(G\)-invariant. Denote by \(b_H\) the block of \(H\) covering \(c\) and covered by \(b\).

The case \(G = N\) is trivial, so we assume \(N < G\) and take a normal subgroup \(M\) of \(G\) such that \(N \leq M < G\) and \(G/M\) is a simple group.
Let \(G[b_M]\) be the subgroup of \(G\) consisting of elements that act as inner automorphisms on the algebra \(kM b_M\). Then \(M \leq G[b_M] \trianglelefteq G\), and thus either \(G[b_M] = G\) or \(G[b_M] = M\).

First, suppose that \(G = G[b_M]\).
Since \(G/M\) is a \(p'\)-group, it follows from \cite[Theorem 7]{Ku} that there exists a simple \(k\)-algebra \(S\) such that \(kGb \cong S \otimes_k kM b_M\). This isomorphism implies that \(\tau(b) = \tau(b_M)\).

Now, suppose that \(M = G[b_M]\).
By \cite[Theorem 3.5]{Da}, \(b\) is the unique block of \(G\) covering \(b_M\), and thus \(\IBr(b) = \bigcup_{\varphi \in \IBr(b_M)} \IBr(G | \varphi)\).
Observe that for any \(\varphi \in \IBr(b_M)\), there exist exactly \(|G:G_{\varphi}|\) irreducible Brauer characters in \(\IBr(b_M)\) which are \(G\)-conjugate to \(\varphi\). Applying Lemma \ref{IBr}, we derive the following
\[
\sum_{\chi \in \IBr(b)} \chi(1)^2 = \sum_{\varphi \in \IBr(b_M)} \frac{1}{|G:G_{\varphi}|} \sum_{\chi \in \IBr(G|\varphi)} \chi(1)^2 = |G:M| \sum_{\varphi \in \IBr(b_M)} \varphi(1)^2.
\]
Combined with the equality \(\dim_k(kGb) = |G:M| \dim_k(kM b_M)\), this yields \(\tau(b) = \tau(b_M)\).

In both cases, we have shown that \(\tau(b) = \tau(b_M)\). By the induction hypothesis, \(\tau(b_M) = \tau(c)\), and therefore \(\tau(b) = \tau(c)\).
\end{proof}

\begin{lem}\label{p-Ex}
Keep the notation above. Let \(P\) be a defect group of \(b\).
If \(G = PN\), then
\[
\frac{\tau(b)}{p^{s(b)}|P|} \leq \frac{\tau(c)}{p^{s(c)}|P \cap N|}.
\]
\end{lem}

\begin{proof}
Since \(P\) is a defect group of \(b\) and \(G= PN\), it follows from \cite[Corollary 9.6 and Theorem 9.14]{N1} that \(b\) is the unique block of \(G\) covering \(c\) and that \(c\) is \(G\)-invariant.
Hence, \(b=c\) and we have
\[
\dim_k(kGb) = |P:P \cap N| \cdot \dim_k(kNc).
\]
By \cite[Theorems 8.9 and 8.11]{N1}, we have
\[
\sum_{\chi \in \IBr(b)} \chi(1)^2 = \sum_{\varphi \in \IBr(c)} |G:G_{\varphi}| \cdot \varphi(1)^2 \geq \sum_{\varphi \in \IBr(c)} \varphi(1)^2.
\]
According to \cite[Theorem 9.26]{N1}, \(P \cap N\) is a defect group of \(c\), which implies \(s(b) \geq s(c)\). Therefore, we have
\[
\frac{\tau(b)}{p^{s(b)}|P|}
= \frac{\dim_k(kGb)}{p^{s(b)}|P| \sum_{\chi \in \IBr(b)} \chi(1)^2}
\leq \frac{|P:P \cap N| \cdot \dim_k(kNc)}{p^{s(b)}|P| \sum_{\varphi \in \IBr(c)} \varphi(1)^2}
\leq \frac{\tau(c)}{p^{s(c)}|P \cap N|}.
\]
This completes the proof.

\end{proof}

\begin{lem}\label{Nil-Ex}
Keep the notation above. Let \(D\) be a defect group of \(c\).
If \(c\) is \(G\)-invariant and nilpotent,
then there exists a finite group \(L\) containing \(D\) as a normal subgroup,
a central \(p'\)-subgroup \(Z\) of \(L\), and a block \(d\) of \(L\) such that:
\begin{enumerate}
\item \(L/DZ \cong G/N\),
\item \(b\) and \(d\) are basic Morita equivalent, and
\item \(\tau(b) = \tau(d)\).
\end{enumerate}
\end{lem}

\begin{proof}
Let \(E = G/N\). By \cite[Theorem 1.1]{CMT}, there exists a group extension \(L'\) of \(E\) by \(D\) and a 2-cocycle \(\alpha \in H^2(E, k^{\times})\) such that there exists a \(E\)-graded basic Morita equivalence \(\Phi'\) between the block extension algebra \(kGc\) and the twisted group algebra \(k_{\alpha}L'\). By \cite[Proposition 1.2.18]{L}, there exists a central \(p'\)-extension \(L\) of \(L'\) by a group \(Z\), together with a central idempotent \(e\) of \(kL\), such that \(kLe \cong k_{\alpha}L'\) as \(k\)-algebras. Here, we regard \(D\) as a normal subgroup of \(L\) through the canonical map \(L \to L'\), thus \(e\) is a block of \(DZ\). Therefore, the \(E\)-graded basic Morita equivalence \(\Phi'\) induces a \(E\)-graded basic Morita equivalence \(\Phi\) between \(kGc\) and \(kLe\).

Let \(d\) be the block of \(L\) corresponding to \(b\) under \(\Phi\) and let \(\sigma\) be the corresponding bijection from \(\IBr(d)\) to \(\IBr(b)\). Thus, \(d\) and \(b\) have the same Cartan matrix, with Cartan elements \(c_{\alpha\beta} = c_{\sigma(\alpha)\sigma(\beta)}\) for \(\alpha,\beta \in \IBr(d)\). Let \(\varphi\) and \(\phi\) be the unique irreducible Brauer characters of the blocks \(c\) and \(e\), respectively. By \cite[Theorem 3.2(6)]{Bo}, we have \(\sigma(\alpha)(1) / \varphi(1) = \alpha(1) / \phi(1)=\alpha(1)\) for all \(\alpha \in \IBr(b)\). Therefore
\[
\tau(b) = \frac{\sum_{\alpha,\beta \in \IBr(d)} c_{\sigma(\alpha)\sigma(\beta)} \sigma(\alpha)(1) \sigma(\beta)(1)}{\sum_{ \xi \in \IBr(d)} \sigma(\xi)(1)^2}
=
\frac{\sum_{\alpha,\beta \in \IBr(d)} c_{\alpha\beta} \alpha(1) \beta(1)}{\sum_{\xi \in \IBr(d)} \xi(1)^2}
= \tau(d).
\]
This completes the proof.
\end{proof}

At the end of this section, we give a proof of Theorem \ref{Main}.
We use the standard notation: \(\mathbf{O}_p(G)\) for the maximal normal \(p\)-subgroup of \(G\),
\(\mathbf{O}_{p'}(G)\) for the maximal normal \(p'\)-subgroup of \(G\),
\(\mathbf{Z}(G)\) for the center of \(G\),
\(\mathbf{E}(G)\) for the layer of \(G\),
\(\mathbf{F}(G)\) for the Fitting subgroup of \(G\),
and
\(\mathbf{F}^*(G)\) for the generalized Fitting subgroup of \(G\).

\begin{proof}[Proof of Theorem \ref{Main}]
Let \(G\) be a finite group and \(b\) a block of \(G\) with non-trivial abelian defect group \(P\).
Suppose that inequality (\ref{equation}) holds for every block \(f\) of every quasi-simple group with non-trivial abelian defect group \(Q\), that is,
\[
\tau(f) < p^{s(f)} |Q|.
\]
We aim to show that \(\tau(b) < p^{s(b)} |P|\).

Note that basic Morita equivalence between blocks preserves defect groups.
Proceeding as in the proof of \cite[Proposition 6.1]{A} and repeatedly applying Lemma \ref{Fong} and Lemma \ref{Nil-Ex} to all blocks of all normal subgroups, we may finally assume the following for any \(N \unlhd G\):
the block of \(N\) covered by \(b\) is \(G\)-invariant;
moreover, if \(b\) covers a nilpotent block of \(N\), then
\(N \leq \mathbf{Z}(G)\mathbf{O}_p(G)\).
In this case, let \(b_N\) denote the unique block of \(N\) covered by \(b\).

Now, suppose that \(\mathbf{E}(G)\) is trivial. Then the generalized Fitting subgroup \(\mathbf{F}^*(G)\) coincides with \(\mathbf{F}(G)\), which implies that \(\mathbf{F}^*(G) = \mathbf{O}_p(G)\mathbf{Z}(G)\).
Given that \(\mathbf{O}_p(G) \leq P\), \(P\) is abelian, and \(C_G(\mathbf{F}^*(G)) \leq \mathbf{F}^*(G)\)(see \cite[Theorem 9.8]{I}), we deduce \(P \leq \mathbf{F}^*(G)\), and hence \(P = \mathbf{O}_p(G)\) and \(C_{G}(P)=P\mathbf{Z}(G)\).
In this case, \(G/P\) is a \(p'\)-group by \cite[Theorem 6.7.6(v)]{L}. Applying Lemma \ref{p-Reg}, we conclude \(\tau(b) = \tau(b_P) = |P| < p^{s(b)} |P|\).

Next, suppose that \(\mathbf{E}(G)\) is non-trivial.
By \cite[Theorem 9.26]{N1}, \(D = P \cap \mathbf{F}^*(G)\) is a defect group of \(b_{\mathbf{F}^*(G)}\).
Set \(K = \mathbf{F}^*(G)C_G(D)\). The Frattini
argument implies that \(K\) is normal in \(G\).
Since \(D \leq P\) and \(P\) is abelian, we have \(C_{G}(P) \leq K\), so \(P\) is a defect group of \(b_K\).
We may take \((P, e_P)\) to be a common maximal Brauer pair associated with the blocks \(b\) and \(b_K\).
By the Frattini argument, \(G = K \cdot N_G(P, e_P)\).
Since \(G/K \cong N_G(P, e_P)/K \cap N_G(P, e_P)\), it follows that \(G/K\) is a \(p'\)-group.
%\(C_G(P) \leq K \cap N_G(P, e_P)\)

Let \(X_1, \dots, X_r\) be all distinct components of \(G\). Notice that \([\mathbf{E}(G),\mathbf{F}(G)]=1\)(see \cite[Theorem 9.7]{I}). For each \(i\), \(X_i\) is normal in \(\mathbf{F}^*(G)\) and let \(d_i\) be the block of \(X_i\) covered by \(b_{\mathbf{F}^*(G)}\).
Suppose, for contradiction, that there exists some \(t\) \((1 \leq t \leq r)\) such that the block \(d_t\) is nilpotent.
Take a maximal subset \(I \subseteq \{1, 2, \ldots, r\}\) with the property that \(d_i\) is nilpotent for all \(i \in I\).
Denote by \(L\) the product of all \(X_i\) with \(i \in I\).
Since \(b_{\mathbf{F}^*(G)}\) is \(G\)-invariant,
then \(L\) is a normal subgroup in \(G\).
By \cite[Lemma 2.7]{Ar}, the block of \(L\) covered by \(b\) is nilpotent, so \(L \subseteq \mathbf{Z}(G)\mathbf{O}_p(G)\). However, this contradicts \(X_t \leq L\), since the component \(X_t\) is not contained in \(\mathbf{F}(G)\).
Therefore, each \(d_i\) is non-nilpotent.

Note that \(D \cap X_i\) is a defect group of \(d_i\), and hence \(D \cap \mathbf{Z}(X_i)\) is properly contained in \(D \cap X_i\).
For any \(x \in C_G(D)\) and any \(i\), we have \(D \cap X_i = (D \cap X_i)^x = D \cap X_i^x\).
This implies \(X_i^x = X_i\) for all \(i\); otherwise, \(D \cap X_i = D \cap \mathbf{Z}(X_i)\), a contradiction.
Thus, every \(X_i\) is normal in \(K\). Since \(K\) is normal in \(G\), by \cite[Problem 9A.1]{I}, we have
\[
\mathbf{F}^*(K) = \mathbf{F}^*(G) = \mathbf{Z}(G)\mathbf{O}_p(G)\mathbf{E}(G) = \mathbf{Z}(K)\mathbf{E}(K).
\]
Moreover, since \(K/\mathbf{F}^*(G)\) is solvable by \cite[Lemma 2.4]{ZZ}, it follows that \(G/\mathbf{E}(G)\) is \(p\)-solvable.
Consider the upper \(p\)-series of \(G/\mathbf{E}(G)\) (see \cite[\S 6.3]{Go})
\[
1 \leq \mathbf{O}_p(G/\mathbf{E}(G)) \leq \mathbf{O}_{p,p'}(G/\mathbf{E}(G)) \leq \mathbf{O}_{p,p',p}(G/\mathbf{E}(G)) \leq \cdots \leq G/\mathbf{E}(G).
\]
Every group in this chain is characteristic in \(G/\mathbf{E}(G)\), and each successive index is either a power of \(p\) or prime to \(p\).
Take a preimage of this series under \(\pi : G \to G/\mathbf{E}(G)\) to obtain
\[
\mathbf{E}(G) = N_0 \leq N_1 \leq N_2 \leq N_3 \leq \cdots \leq N_t = G,
\]
where each subgroup \(N_i\) is characteristic in \(G\). Let \(b_{N_i}\) be the unique blocks of \(N_i\) covered by \(b\), and let \(P_i = P \cap N_i\) be the defect group of \(b_{N_i}\).
By alternating applications of Lemma \ref{p-Reg} and Lemma \ref{p-Ex}, we obtain
\[
\frac{\tau(b)}{p^{s(b)}|P|} = \frac{\tau(b_{N_t})}{p^{s(b_{N_t})}|P_t|} \leq \frac{\tau(b_{N_{t-1}})}{p^{s(b_{N_{t-1}})}|P_{t-1}|} \leq \cdots \leq \frac{\tau(b_{N_0})}{p^{s(b_{N_0})}|P_0|} = \frac{\tau(b_{\mathbf{E}(G)})}{p^{s(b_{\mathbf{E}(G)})}|P \cap \mathbf{E}(G)|}.
\]

Recall that \(\mathbf{E}(G) = X_1 X_2 \cdots X_r\) is a central product of \(X_i\).
Let \(\bar{\mathbf{E}}(G) = \mathbf{E}(G)/\mathbf{O}_p(\mathbf{Z}(\mathbf{E}(G)))\) and let \(\bar{X}_i\) be the image of \(X_i\) in \(\bar{\mathbf{E}}(G)\).
Let \(\bar{b}_{\mathbf{E}(G)}\) and \(\bar{d}_i\) be the blocks of \(\bar{\mathbf{E}}(G)\) and \(\bar{X}_i\) dominated by \(b_{\mathbf{E}(G)}\) and \(d_i\), respectively.
Let \(\bar{D}\) and \(\bar{D}_i\) be the images of \(P \cap \mathbf{E}(G)\) and \(P \cap X_i\) in \(\bar{\mathbf{E}}(G)\), respectively.
Then \(\bar{D}\) and \(\bar{D}_i\) are defect groups of \(\bar{b}_{\mathbf{E}(G)}\) and \(\bar{d}_i\), respectively.
Applying  \cite[Lemma 7.5(i,ii)]{S}, we have
 \[\bar{b}_{\mathbf{E}(G)}=\bar{d}_1\otimes \cdots\otimes \bar{d}_r~{\rm and}~\bar{D}=\bar{D}_1\times\cdots\times \bar{D}_r.\]
A direct computation shows that \[\frac{\tau(\bar{b}_{\mathbf{E}(G)})}{p^{s(\bar{b}_{\mathbf{E}(G)})}|\bar{D}|}
        =\prod_{i=1}^r\frac{\tau(\bar{d}_i)}{p^{s(\bar{d}_i)}|\bar{D}_i|}.\]
Since each \(d_i\) is non-nilpotent, so is \(\bar{d}_i\) by \cite[Lemma 2]{W}.
As \(\bar{X}_i\) is quasi-simple, our assumption implies
\[
\frac{\tau(\bar{d}_i)}{p^{s(\bar{d}_i)}|\bar{D}_i|} < 1.
\]
Summarizing the above, by Lemma \ref{Central}, we have
\[\frac{\tau(b)}{p^{s(b)}|P|}\leq\frac{\tau(b_{\mathbf{E}(G)})}{p^{s(\mathbf{E}(G))}|P\cap \mathbf{E}(G)|}\leq\frac{\tau(\bar{b}_{\mathbf{E}(G)})}{p^{s(\bar{b}_{\mathbf{E}(G)})}|\bar{D}|}
        =\prod_{i=1}^r\frac{\tau(\bar{d}_i)}{p^{s(\bar{d}_i)}|\bar{D}_i|}<1.\]
This completes the proof.
\end{proof}

\section{Proof of  Corollary \ref{Main1}}

In this section, we give a proof of Corollary \ref{Main1}. Throughout, we fix the prime \(p = 2\). First, we recall the classification of 2-blocks of quasi-simple groups with abelian defect groups.

\begin{thm}{{\rm (\cite[Theorem 6.1]{EKK})}}\label{Classification}
Let \(G\) be a quasi-simple group and \(b\) a block of \(G\) with abelian defect group \(P\). Then one of the following holds
\begin{enumerate}[itemsep=2pt,topsep=5pt,label=\emph{(\roman*)}]
\item \(G/Z(G)\) is one of \(A_1(2^n)\), \(^2G_2(q)\) (where \(q \geq 27\) is a power of \(3\) with odd exponent), or \(J_1\), \(b\) is the principal block and \(P\) is elementary abelian.

\item \(G\) is \(Co_3\), \(b\) is a non-principal block and \(P\) is an elementary abelian group of rank \(3\) (there is one such block).

\item There exists a finite group \(\widetilde{G}\) such that \(G \unlhd \widetilde{G}\), \(Z(G) \leq Z(\widetilde{G})\) and such that \(b\) is covered by a nilpotent block of \(\widetilde{G}\).

\item The block \(b\) is Morita equivalent to a block \(b_L\) of \(L\) where \(L = L_0 \times L_1\) is a subgroup of \(G\) such that the following holds: The defect groups of \(b_L\) are isomorphic to \(P\), \(L_0\) is abelian, and the block of \(L_1\) covered by \(b_L\) has Klein four defect groups.
\end{enumerate}
\end{thm}

\begin{prop}\label{Car}
Let \(G\) be a quasi-simple group and \(b\) a block of \(G\) with abelian defect group \(P\).
Then for every \(\varphi \in \IBr(b)\), we have
\(
c_{\varphi\varphi} \leq |P|,
\)
where \(c_{\varphi\varphi}\) is the corresponding diagonal entry of the Cartan matrix of \(b\).
\end{prop}
\begin{proof}
The proof proceeds by examining the classification of such blocks given in Theorem~\ref{Classification}.

Suppose that \( G/Z(G) \) is isomorphic to \( A_1(2^n) \) or \( ^2G_2(q) \), where \( q \geq 27 \) is a power of \( 3 \) with odd exponent, and that \( b \) is the principal block of \( G \). Then, by \cite[Theorem 2]{Al} and \cite[Theorem 3.9]{La}, we have \( c_{\varphi\varphi} \leq |P| \) for every \( \varphi \in \IBr(b) \).

 Suppose that either \( G \) is isomorphic to \( J_1 \) and \( b \) is its principal block, or \( G \) is isomorphic to \( Co_3 \) and \( b \) is the block stated in Theorem~\ref{Classification}(ii). In both cases, the defect group \( P \) of \( b \) is elementary abelian of order \( 8 \). Using the GAP Character Table Library \cite{GAP2020}, we compute the Cartan matrices for \( b \):
\begin{enumerate}[itemsep=2pt, topsep=5pt, label=(\roman*)]
    \item For \( G \cong J_1 \):
    \[
    C = \begin{pmatrix}
        8 & 4 & 4 & 4 & 4 \\
        4 & 4 & 3 & 3 & 1 \\
        4 & 3 & 4 & 2 & 2 \\
        4 & 3 & 2 & 4 & 2 \\
        4 & 1 & 2 & 2 & 4
    \end{pmatrix}.
    \]
    \item For \( G \cong Co_3 \):
    \[
    C = \begin{pmatrix}
        4 & 2 & 4 & 2 & 2 \\
        2 & 4 & 4 & 2 & 2 \\
        4 & 4 & 8 & 4 & 3 \\
        2 & 2 & 4 & 4 & 2 \\
        2 & 2 & 3 & 2 & 2
    \end{pmatrix}.
    \]
\end{enumerate}
A direct inspection of the diagonal entries shows that \( c_{\varphi\varphi} \leq 8 = |P| \) for all \( \varphi \in \IBr(b) \) in both cases.

Suppose that \(b\) is as stated in Theorem~\ref{Classification}(iii). By \cite[Theorem 3.13]{P2}, the block \(b\) is inertial, and its block algebra \(kGb\) is Morita equivalent to a twisted group algebra \(k_\alpha (P \rtimes E)\), where \(E\) is a \(2'\)-group acting on \(P\) and \(\alpha \in H^2(E, k^\times)\). Furthermore, by \cite[Proposition 1.2.18]{L}, there exists a central \(2'\)-extension \(L\) of \(P \rtimes E\) by a group \(Z\) together with a central idempotent \(d\) of \(kL\) such that \(k_\alpha (P \rtimes E) \cong kL d\) as \(k\)-algebras. Since \(Z\) is a \(2'\)-group, we have \(|L|_2 = |P|\). By \cite[Theorem 2.22 and Corollary 10.14]{N1}, it follows that \(c_{\phi\phi} \leq |P|\) for all \(\phi \in \IBr(d)\). The Morita equivalence between \(kGb\) and \(kLd\) implies that the same inequality holds for all \( \varphi \in \IBr(b)\).

Suppose that \(b\) is a non-nilpotent block satisfying the conditions in Theorem~\ref{Classification}(iv), and keep the notation therein. Since \(b\) is non-nilpotent, the block \(b_{L_1}\) of \(L_1\) covered by \(b_L\) is also non-nilpotent.

As \(b_{L_1}\) has Klein four defect groups, counting the number of composition factors of the projective indecomposable modules (see \cite[Theorem 4]{E}) shows the Cartan matrix \(C_{b_{L_1}}\) of \(b_{L_1}\) must be one of the following:
\[
\begin{pmatrix}
2 & 1 & 1 \\
1 & 2 & 1 \\
1 & 1 & 2
\end{pmatrix}
\quad \text{or} \quad
\begin{pmatrix}
4 & 2 & 2 \\
2 & 2 & 1 \\
2 & 1 & 2
\end{pmatrix}.
\]
By \cite[Theorem 9.10]{N1}, the Cartan matrix of \(b_L\) is equal to \(|L_0|_2 \cdot C_{b_{L_1}}\). Since \(|P| = 4 \cdot |L_0|_2\), we have \(c_{\phi\phi} \leq |P|\) for every \(\phi \in \IBr(b_L)\).
The Morita equivalence between \(kGb\) and \(kLb_L\) from Theorem~\ref{Classification}(iv) implies the same inequality holds for all \(\varphi \in \IBr(b)\).

Having covered all cases in Theorem~\ref{Classification}, the proposition is proved.
\end{proof}

\begin{proof}[Proof of Corollary \ref{Main1}]
By Theorem \ref{Main}, to prove Corollary \ref{Main1}, it suffices to verify that inequality (\ref{equation}) holds for all blocks of quasi-simple groups with non-trivial abelian defect groups.
Let \(G\) be a quasi-simple group and \(b\) a block of \(G\) with non-trivial abelian defect group \(P\).
We aim to show that \(\tau(b) < p^{s(b)}|P|\).

Let \(C\) be the Cartan matrix of \(b\) with trace \(\mathrm{Tr}(C)\). By \cite[Proposition 2]{HW}, we have \(\tau(b) \leq \mathrm{Tr}(C)\). Since \(P\) is non-trivial, by \cite[Theorem G]{R} and \cite[Corollary 3]{MR}, we obtain \(l(b) < p^{s(b)}\). Furthermore, Proposition~\ref{Car} establishes that \(c_{\varphi\varphi} \leq |P|\) for every \(\varphi \in \IBr(b)\). Consequently, \[\mathrm{Tr}(C) = \sum_{\varphi \in \IBr(b)} c_{\varphi\varphi} \leq l(b) \cdot |P| < p^{s(b)} |P|.\] Therefore, \(\tau(b) \leq \mathrm{Tr}(C) < p^{s(b)} |P|\), as required.
\end{proof}

%%%%%%%%%%%%%%%%%%%%%%%%%%%%%%%%%%%%%%%%%%%%%%%%%%%%%%%
%%%%%%%%%%%%%%%%%%%%%%%%%%%%%%%%%%%%%%%%%%%%%%%%%%%%%%%

\end{document}